\documentclass[10pt]{amsart}
\usepackage{amssymb}
\usepackage{epsfig}
\usepackage{url}
\usepackage{setspace}
\usepackage{color}
\theoremstyle{plain}

\newtheorem{thm}{Theorem}[section]
\newtheorem{cor}[thm]{Corollary}
\newtheorem{lem}[thm]{Lemma}
\newtheorem{prop}[thm]{Proposition}

\renewcommand{\phi}{\varphi}

\newcommand{\Diag}{{\rm Diag}\mathbb{P}}

\begin{document}
\title[A note on primes with prime indices]{A note on primes with prime indices}
\author{B{\l}a\.{z}ej \.{Z}mija}

\thanks{During the preparation of the work, the author was a scholarship holder of the Kartezjusz program funded by the Polish National Center for Research and Development.}

\begin{abstract}
Let $n,k\in\mathbb{N}$ and let $p_{n}$ denote the $n$th prime number. We define $p_{n}^{(k)}$ recursively as $p_{n}^{(1)}:=p_{n}$ and $p_{n}^{(k)}=p_{p_{n}^{(k-1)}}$, that is, $p_{n}^{(k)}$ is the $p_{n}^{(k-1)}$th prime. 

In this note we give answers to some questions and prove a conjecture posed by Miska and T\'{o}th in their recent paper concerning subsequences of the sequence of prime numbers. In particular, we establish explicit upper and lower bounds for $p_{n}^{(k)}$.
We also study the behaviour of the counting functions of the sequences $(p_{n}^{(k)})_{k=1}^{\infty}$ and $(p_{k}^{(k)})_{k=1}^{\infty}$.
\end{abstract}

\maketitle

\section{Introduction}

Let $(p_{n})_{n=1}^{\infty}$ be the sequence of consecutive prime numbers. In a recent paper \cite{MT} Miska and T\'{o}th introduced the following subsequences of the sequence of prime numbers: $p_{n}^{(1)}:=p_{n}$ and for $k\geq 2$
\begin{align*}
    p_{n}^{(k)}:=p_{p_{n}^{(k-1)}}.
\end{align*}
In other words, $p_{n}^{(k)}$ is the $p_{n}^{(k-1)}$th prime. They also defined
\begin{align*}
    \Diag := & \{\ p_{k}^{(k)}\ |\ k\in\mathbb{N}\ \}, \\
    \mathbb{P}_{n}^{T}:= & \{\ p_{n}^{(k)}\ |\ k\in\mathbb{N}\ \}
\end{align*}
for each positive integer $n$.

The main motivation in \cite{MT} was the known result that the set of prime numbers is ($R$)-dense, that is, the set $\{\ \frac{p}{q}\ |\ p,q\in\mathbb{P}\ \}$ is dense in $\mathbb{R}_{+}$ (with respect to the natural topology on $\mathbb{R}_{+}$). It was proved in \cite{MT} that for each $k\in\mathbb{N}$ the sequence $\mathbb{P}_{k}:=(p_{n}^{(k)})_{n=1}^{\infty}$ is ($R$)-dense. This result might be surprising, because the sequences $\mathbb{P}_{k}$ are very sparse. In fact, for each $k$ set $\mathbb{P}_{k+1}$ is a zero asymptotic density subset of $\mathbb{P}_{k}$. On the other hand, it was showed, that the sequences $(p_{n}^{(k)})_{k=1}^{\infty}$ for each fixed $n\in\mathbb{N}$, and $(p_{k}^{(k)})_{k=1}^{\infty}$ are not ($R$)-dense.

Results of another type that were proved in \cite{MT} concern the asymptotic behaviour of $p_{n}^{(k)}$ as $n\rightarrow\infty$, or as $k\rightarrow\infty$. In particular, as $n\rightarrow\infty$, we have for each $k\in\mathbb{N}$
\begin{align*}
    p_{n}^{(k)}\sim n(\log n)^{k},\ \ \ \ p_{n+1}^{(k)}\sim p_{n}^{(k)}, \ \ \ \ \log p_{n}^{(k)}\sim \log n
\end{align*}
by \cite[Theorem 1]{MT}. Some results from \cite{MT} concerning $p_{n}^{(k)}$ as $k\rightarrow\infty$ are mentioned later.

For a set $A\subseteq \mathbb{N}$ let $A(x)$ be its counting function, that is,
\begin{align*}
    A(x):=\# \left(A\cap [1,x]\right).
\end{align*}
Miska and T\'{o}th posed four questions concerning the numbers $p_{n}^{(k)}$:
\begin{enumerate}
    \item[A.] Is it true that $p_{k+1}^{(k)}\sim p_{k}^{(k)}$ as $k\rightarrow\infty$?
    \item[B.] Are there real constants $c>0$ and $\beta$ such that
    \begin{align*}
        \exp\mathbb{P}_{n}^{T}(x)\sim cx(\log x)^{\beta}
    \end{align*}
    for each $n\in\mathbb{N}$?
    \item[C.] Are there real constants $c>0$ and $\beta$ such that
    \begin{align*}
        \exp\Diag (x)\sim cx(\log x)^{\beta}?
    \end{align*}
    \item[D.] Is it true that
    \begin{align*}
        \Diag (x)\sim\mathbb{P}_{n}^{T}(x)
    \end{align*}
    for each $n\in\mathbb{N}$?
\end{enumerate}
The aim of this paper it to give answers to question B, C and D. 

The main ingredients of our proofs are the following inequalities:
\begin{align}\label{ineqPrimes}
    n\log n<p_{n}<2n\log n.
\end{align}
The first inequality holds for all $n\geq 2$, and the second one for all $n\geq 3$. For the proofs, see \cite{RS}. In Section \ref{Results} we use (\ref{ineqPrimes}) in order to show explicit bounds for $p_{n}^{(k)}$. In particular, for all $n>e^{4200}$ we have:
\begin{align*}
    \log p_{n}^{(k)}= & k(\log k+\log\log k+O_{n}(1)), \\
    \log p_{k}^{(k)}= & k(\log k+\log\log k+O(\log\log\log k)),
\end{align*}
as $k\rightarrow\infty$, where the implied constant in the first line may depend on $n$, see Theorem \ref{MAIN} below. In consequence, we improve the (in)equalities
\begin{align*}
    \lim_{k\rightarrow\infty}\frac{p_{n}^{(k)}}{k\log k}= & 1, \\
    1\leq \liminf_{k\rightarrow\infty}\frac{p_{k}^{(k)}}{k\log k}\leq &  \limsup_{k\rightarrow\infty}\frac{p_{k}^{(k)}}{k\log k}\leq 2.
\end{align*}
that appeared in \cite{MT}. Then we show in Section \ref{Coro} that the answers to questions B and C are negative (Corollary \ref{CoroBC}), while the one for question D is affirmative (Theorem \ref{AsympEqThm}). In fact, we find the following relation:
\begin{align*}
    \mathbb{P}_{n}^{T}(x)\sim\Diag (x)\sim\frac{\log x}{\log\log x}
\end{align*}
for all positive integers $n$.

In their paper, Miska and T\'{o}th also posed a conjecture, that we state here as a proposition, since it is in fact a consequence of a result that had already appeared in \cite{MT}.

\begin{prop}
Let $n\in\mathbb{N}$ be fixed. Then
\begin{align*}
    \frac{p_{n}^{(k)}}{p_{k}^{(k)}}\longrightarrow 0
\end{align*}
as $k\longrightarrow\infty$.
\end{prop}
\begin{proof}
Let $k> p_{n}$. Then
\begin{align*}
    0\leq \frac{p_{n}^{(k)}}{p_{k}^{(k)}}<\frac{p_{n}^{(k)}}{p_{p_{n}}^{(k)}}=\frac{p_{n}^{(k)}}{p_{n}^{(k+1)}}.
\end{align*}
The expression on the right goes to zero as $k$ goes to infinity, as was proved in \cite[Corollary 3]{MT}.
\end{proof}

It is worth to note, that primes with prime indices have already appeared in the literature, for example in \cite{BB} and \cite{BKO}. However, according to our best knowledge, our paper is the second one (after \cite{MT}), where the number of iterations of indices, that is, the number $k$ in $p_{n}^{(k)}$, is not fixed.

Throughout the paper we use the following notation: $\log x$ denotes the natural logarithm of $x$, and for functions $f$ and $g$ we write $f\sim g$ if $\lim_{x\rightarrow\infty}\frac{f(x)}{g(x)}=1$, $f=O(g)$ if there exists a positive constants $c$ such that $f(x)<cg(x)$, and $f=o(g)$ if $\lim_{x\rightarrow\infty}\frac{f(x)}{g(x)}=0$.

\section{Upper and lower bounds for $p_{n}^{(k)}$}\label{Results}

In this Section, we find explicit upper and lower bounds for $p_{n}^{(k)}$. We start with the upper bound.

\begin{lem}\label{lemUP}
Let $n\geq 9$. Then for each $k\in\mathbb{N}$ we have:
\begin{align*}
    p_{n}^{(k)}<2^{2k-1}\cdot n\cdot (k-1)!\cdot\big(\log(\max\{k,n\})\big)^{k}. 
\end{align*}
In particular,
\begin{align*}
    p_{n}^{(k)}< \big(4\cdot k\log k\big)^{k}
\end{align*}
for $k\geq n$.
\end{lem}
\begin{proof}
We proceed by induction on $k$. For $k=1$ it is a simple consequence of (\ref{ineqPrimes}).
Then the second induction step goes as follows: let us denote $m:=\max\{n,k\}$. Observe that $(k-1)!<(k-1)^{k-1}<m^{k-1}$ and $4\log m<m$ for $m\geq 9$. Hence,
\begin{align*}
    p_{n}^{(k+1)}\leq & 2p_{n}^{k}\log p_{n}^{(k)}<2\cdot 2^{2k-1}\cdot n\cdot (k-1)!\cdot(\log m)^{k}\cdot \log\left[2^{2k-1}\cdot n\cdot (k-1)!\cdot (\log m)^{k}\right] \\
    < & 2^{2k}\cdot n\cdot (k-1)!\cdot(\log m)^{k} \log\left[4^{k}\cdot m\cdot m^{k-1}\cdot (\log m)^{k}\right]=2^{2k}\cdot n\cdot (k-1)!\cdot(\log m)^{k} \log\left[4\cdot m\cdot \log m\right]^{k} \\
    \leq & 2^{2k}\cdot n\cdot k!\cdot (\log m)^{k}\cdot \log[m]^{2}\leq 2^{2k+1}\cdot n\cdot k!\cdot (\log m)^{k+1}.
\end{align*}

The second part of the statement is an easy consequence of the first part and the inequalities $(k-1)!<k^{k-1}$ and $n\leq k$.
\end{proof}

In order to prove a lower bound for $p_{n}^{(k)}$ we will need the following fact.

\begin{lem}\label{lemL}
Let
\begin{align*}
    L(x):=\left(\frac{x}{x+1}\right)^{x+1}\left(\frac{\log x}{\log (x+1)}\right)^{x+1}.
\end{align*}
Then we have
\begin{align*}
    L(x)>0.32627
\end{align*}
for all $x\geq 4200$.
\end{lem}
\begin{proof}
Observe, that the function $\left(\frac{x}{x+1}\right)^{x+1}$ is increasing. Indeed, if
\begin{align*}
    f(x):=\log\left(\frac{x}{x+1}\right)^{x+1}=(x+1)\left(\log x-\log (x+1)\right),
\end{align*}
then
\begin{align*}
    f'(x)=\log x-\log (x+1)+(x+1)\left(\frac{1}{x}-\frac{1}{x+1}\right)=\frac{1}{x}-\log \left(1+\frac{1}{x}\right)>0,
\end{align*}
where the last inequality follows from the well-known inequality $y>\log (1+y)$ used with $y=\frac{1}{x}$. Hence, we can bound
\begin{align}\label{L1}
    \left(\frac{x}{x+1}\right)^{x+1}\geq \left(\frac{4200}{4201}\right)^{4201}
\end{align}
for all $x\geq 4200$.

Now we need to find a lower bound for $\left(\frac{\log x}{\log (x+1)}\right)^{x+1}$. Let us write
\begin{align*}
    \left(\frac{\log x}{\log (x+1)}\right)^{x+1}=\left[\left(1-\frac{1}{\frac{\log (x+1)}{\log (x+1)-\log x}}\right)^{\frac{\log (x+1)}{\log (x+1)-\log x}}\right]^{\log\left(1+\frac{1}{x}\right)^{x}\cdot\frac{1}{\log (x+1)}\cdot\left(1+\frac{1}{x}\right)}.
\end{align*}
At first, we prove that functions $g(t):=\left(1-\frac{1}{t}\right)^{t}$ and $h(x):=\frac{\log (x+1)}{\log (x+1)-\log x}$ are increasing. For the function $g(t)$ it is enough to observe, that $\log g(t)=f(t-1)$ and the function $f(x)$ is increasing. For the function $h(x)$ we have:
\begin{align*}
    h'(x)= & \frac{1}{(\log (x+1)-\log x)^{2}}\left[\frac{\log (x+1)-\log x}{x+1}-\log (x+1)\left(\frac{1}{x+1}-\frac{1}{x}\right)\right] \\
    = & \frac{1}{(\log (x+1)-\log x)^{2}}\left[\frac{\log (x+1)}{x}-\frac{\log x}{x+1}\right] >0.
\end{align*}

The fact that the functions $g(t)$ and $h(x)$ are increasing, together with the properties $g(h(4200))\in (0,1)$ and $\log\left(1+\frac{1}{x}\right)^{x}<1$, give us
\begin{equation}\label{L2}
    \begin{aligned}
        \left(\frac{\log x}{\log (x+1)}\right)^{x+1}\geq & \big[g(h(4200))\big]^{\log\left(1+\frac{1}{x}\right)^{x}\cdot\frac{1}{\log (x+1)}\cdot\left(1+\frac{1}{x}\right)} 
        > \big[g(h(4200))\big]^{\frac{1}{\log (x+1)}\cdot\left(1+\frac{1}{x}\right)} \\
        \geq & \big[g(h(4200))\big]^{\frac{1}{\log (4200+1)}\cdot\left(1+\frac{1}{4200}\right)}=\left(\frac{\log 4200}{\log 4201}\right)^{\frac{4201}{4200}\cdot\frac{1}{\log 4201 -\log 4200}}
    \end{aligned}
\end{equation}
for all $x\geq 4200$.

Combining (\ref{L1}) and (\ref{L2}) we get the inequality
\begin{align*}
    L(x)\geq \left(\frac{4200}{4201}\right)^{4201}\left(\frac{\log 4200}{\log 4201}\right)^{\frac{4201}{4200}\cdot\frac{1}{\log 4201 -\log 4200}}\approx 0.3262768>0.32627.
\end{align*}
The proof is finished.
\end{proof}

In the next lemma we provide a lower bound for $p_{n}^{(k)}$.

\begin{lem}\label{lemDOWN}
If $n>e^{4200}$, then for all $k\geq\lfloor\log n\rfloor$ we have
\begin{align*}
    p_{n}^{(k)}>\left(\frac{e\cdot k\log k}{\log\log n}\right)^{k}.
\end{align*}
\end{lem}
\begin{proof}
First, let us observe that a simple induction argument on $k$ implies the inequality
\begin{align}\label{ineqDOWN}
    p_{n}^{(k)}>n(\log n)^{k}.
\end{align}
Indeed, for $k=1$ this follows from left inequality in (\ref{ineqPrimes}). Using the same inequality we get also 
\begin{align*}
p_{n}^{(k+1)}>p_{n}^{(k)}\log p_{n}^{(k)}>n (\log n)^{k}\log(n (\log n)^{k})>n(\log n)^{k+1},
\end{align*}
and hence (\ref{ineqDOWN}).

Now we show that the inequality from the statement is true for $k=\lfloor \log n\rfloor$. Because of (\ref{ineqDOWN}) it is enough to show:
\begin{align*}
    n(\log n)^{k}>\left(\frac{e\cdot k\log k}{\log\log n}\right)^{k},
\end{align*}
or equivalently, after taking logarithms we get
\begin{align*}
    \log n+\lfloor\log n\rfloor\log\log n>\lfloor\log n\rfloor+\lfloor\log n\rfloor\log\lfloor\log n\rfloor+\lfloor\log n\rfloor\log\log\lfloor\log n\rfloor -\lfloor\log n\rfloor\log\log\log n.
\end{align*}
This is equivalent to the inequality
\begin{align*}
    \big(\log n-\lfloor\log n\rfloor\big)+\lfloor\log n\rfloor\big(\log\log n-\log\lfloor\log n\rfloor\big)+\lfloor\log n\rfloor\big(\log\log\log n-\log\log\lfloor\log n\rfloor\big)>0,
\end{align*}
which is obviously true.

In order to finish the proof, we again use the induction argument. The inequality from the statement of our lemma is true for $k=\lfloor \log n\rfloor$. Assume it holds for some $k\geq \lfloor\log n\rfloor$. Then by (\ref{ineqPrimes}) and the induction hypothesis we get
\begin{align*}
    p_{n}^{(k+1)}> & p_{n}^{(k)}\log p_{n}^{(k)}>\left(\frac{e\cdot k\log k}{\log\log n}\right)^{k}\log \left(\frac{e\cdot k\log k}{\log\log n}\right)^{k}.
\end{align*}
It is enough to show that for all $n>e^{4200}$ and all $k\geq\lfloor\log n\rfloor$ we have
\begin{align*}
    \left(\frac{e\cdot k\log k}{\log\log n}\right)^{k}\log \left(\frac{e\cdot k\log k}{\log\log n}\right)^{k}>\left(\frac{e\cdot (k+1)\log (k+1)}{\log\log n}\right)^{k+1}.
\end{align*}
This is equivalent to
\begin{align*}
    k^{k+1}(\log k)^{k}\left[\log k+\log\left(\frac{e\log k}{\log\log n}\right)\right]>\frac{e}{\log\log n}(k+1)^{k+1}(\log(k+1))^{k+1}.
\end{align*}
Recall, that we assume that $k\geq \lfloor\log n\rfloor$. Thus $k^{e}>\log n$, that is, $e\log k>\log\log n$. Therefore, it is enough to show the following inequalities:
\begin{align*}
    k^{k+1}(\log k)^{k+1}>\frac{e}{\log\log n}(k+1)^{k+1}(\log(k+1))^{k+1},
\end{align*}
or equivalently
\begin{align*}
    \left(\frac{k}{k+1}\right)^{k+1}\left(\frac{\log k}{\log (k+1)}\right)^{k+1}>\frac{e}{\log\log n}.
\end{align*}
Notice that the left-hand side expression of the last inequality is equal to $L(k)$, where the function $L(x)$ is defined in the statement of Lemma \ref{lemL}. If $n>e^{4200}$, then $k\geq \lfloor\log n\rfloor \geq 4200$, and Lemma \ref{lemL} implies $L(k)>0.32627$. Therefore, if $N:=\max\left\{\lfloor e^{4200}\rfloor,\lceil e^{e^{e/0.32627}}\rceil\right\}=\lfloor e^{4200}\rfloor$, then for all $n>N$ and $k\geq \lfloor\log n\rfloor$ we have:
\begin{align*}
    L(k)>0.32627\geq \frac{e}{\log\log N}>\frac{e}{\log\log n}.
\end{align*}
This finishes the proof.
\end{proof}

\section{Main results}\label{Coro}

We begin this section by a theorem that provides good information about asymptotic growth of $\log p_{n}^{(k)}$ for large fixed $n$, and for $\log p_{k}^{(k)}$ as $k\rightarrow\infty$.

\begin{thm}\label{MAIN}
\begin{enumerate}
    \item Let $n>e^{4200}$. Then
    \begin{align*}
        \log p_{n}^{(k)}=k(\log k+\log\log k+O_{n}(1))
    \end{align*}
    as $k\rightarrow\infty$, where the implied constant may depend on $n$.
    \item We have
    \begin{align*}
        \log p_{k}^{(k)}=k(\log k+\log\log k+O(\log\log\log k))
    \end{align*}
    as $k\rightarrow\infty$.
\end{enumerate}
\end{thm}
\begin{proof}
If $n>e^{4200}$ and $k\geq n$, Lemmas \ref{lemUP} and \ref{lemDOWN} give us:
\begin{align*}
    \left(\frac{e\cdot k\log k}{\log\log n}\right)^{k}<p_{n}^{(k)}<\left(4\cdot k\log k\right)^{k}.
\end{align*}
After taking logarithms, we simply get the first part of our theorem. In order to get the second part, we need to put $n=k$ and repeat the reasoning.
\end{proof}

Now we give the answer to Question D.

\begin{thm}\label{AsympEqThm}
For each $n\in\mathbb{N}$ we have
\begin{align*}
    \Diag (x)\sim \mathbb{P}_{n}^{T} (x)
\end{align*}
as $x\rightarrow\infty$.
\end{thm}
\begin{proof}
From \cite[Theorem 6]{MT} we know that
\begin{align*}
    \mathbb{P}_{m}^{T}(x)\sim \mathbb{P}_{n}^{T}(x)
\end{align*}
for each $m,n\in\mathbb{N}$. Therefore, it is enough to prove $\Diag (x)\sim \mathbb{P}_{n}^{T} (x)$ for some sufficiently large $n$.

Let $n=\lfloor e^{4200}\rfloor +100$. We use the idea from the proof of \cite[Theorem 17]{MT}. Let $x$ be a large real number. Let $k$ be such that $p_{k}^{(k)}\leq x<p_{k+1}^{(k+1)}$. Then $\Diag (x)=k$. By \cite[Theorem 8]{MT} and Theorem \ref{MAIN} above we have
\begin{align*}
    \frac{\mathbb{P}_{n}^{T}(x)}{\Diag (x)}\leq & 1+\frac{\log p_{k+1}^{(k+1)}}{k\log\log p_{n}^{(k)}}-\frac{\log p_{n}^{(k)}}{k\log\log p_{n}^{(k)}} \\
    = & 1+\frac{(1+o(1))(k+1)\log (k+1)}{k\log \big[(1+o(1))k\log k\big]}-\frac{(1+o(1))k\log k}{k\log\big[(1+o(1))k\log k\big]} \\
    = & 1+(1+o(1))\left(1+\frac{1}{k}\right)\frac{\log (k+1)}{\log k+\log\left[(1+o(1))\log k\right]}-(1+o(1))\frac{\log k}{\log k+\log\left[(1+o(1))\log k\right]}.
\end{align*}
The whole last expression goes to $1$ as $k$ goes to infinity. On the other hand, $\Diag (x)\leq \mathbb{P}_{n}^{T}(x)$ for $x\geq p_{n}^{(n)}$ and we get the result.
\end{proof}

The answers to Questions B and C will follow from our next result, which is of independent interest.

\begin{thm}\label{AsympDiagP}
\begin{enumerate}
    \item Let $n\in\mathbb{N}$. Then
    \begin{align*}
        \mathbb{P}_{n}^{T}(x)\sim\frac{\log x}{\log\log x}.
    \end{align*}
    \item We have
    \begin{align*}
        \Diag (x)\sim \frac{\log x}{\log\log x}.
    \end{align*}
\end{enumerate}
\end{thm}
\begin{proof}
In view of Theorem \ref{AsympEqThm}, it is enough to show the statement for the function $\Diag (x)$. Let us fix an arbitrarily small number $\varepsilon >0$ and take a sufficiently large real number $x$ and find $k$ such that $p_{k}^{(k)}\leq x<p_{k+1}^{(k+1)}$. Then $\Diag (x)=k$ and by Lemmas \ref{lemUP} and \ref{lemDOWN} we have
\begin{align*}
    k^{k}<x<k^{(1+\varepsilon )k}.
\end{align*}
Let us write $x=e^{y}$. Then 
\begin{align*}
    k\log k< y<(1+\varepsilon )k\log k.
\end{align*}
If $y$ is sufficiently large, this implies 
\begin{align}\label{ineqyk}
    (1-\varepsilon )\frac{y}{\log y}< k<(1+\varepsilon )\frac{y}{\log y}.
\end{align}
Indeed, if $k\leq (1-\varepsilon )\frac{y}{\log y}$, then
\begin{align*}
    y<(1+\varepsilon)k\log k\leq (1-\varepsilon^{2})\frac{y}{\log y}\log\left((1-\varepsilon )\frac{y}{\log y}\right)<(1-\varepsilon^{2})\left(1-\frac{\log\log y}{\log y}\right)y,
\end{align*}
which is impossible. Similarly, if $k\geq (1+\varepsilon )\frac{y}{\log y}$, then
\begin{align*}
    y>k\log k\geq (1+\varepsilon )\frac{y}{\log y}\log\left((1+\varepsilon )\frac{y}{\log y}\right)>(1+\varepsilon )\left(1-\frac{\log\log y}{\log y}\right)y.
\end{align*}
The above inequality cannot hold if $y$ is sufficiently large.

If we go back to $k=\Diag (x)$ and $y=\log x$ in (\ref{ineqyk}), we get
\begin{align*}
    (1-\varepsilon )\frac{\log x}{\log\log x}<\Diag (x)<(1+\varepsilon )\frac{\log x}{\log\log x}.
\end{align*}
The number $\varepsilon >0$ was arbitrary, so the result follows.
\end{proof}

\begin{cor}\label{CoroBC}
There do not exist constants $c>0$ and $\beta$ such that
\begin{align*}
    \exp\mathbb{P}_{n}^{T}(x)\sim cx(\log x)^{\beta}
\end{align*}
for some $n$, or
\begin{align*}
    \exp\Diag (x)\sim cx(\log x)^{\beta}.
\end{align*}
\end{cor}
\begin{proof}
We prove the result only for $\Diag (x)$. The case of $\mathbb{P}_{n}^{T}(x)$ is analogous.

Assume to the contrary, that $\exp\Diag (x)\sim cx(\log x)^{\beta}$ for some $c>0$ and $\beta$. Then $\exp\Diag (x)=(1+o(1))cx(\log x)^{\beta}$. This, after taking logarithms on both sides, implies $\Diag (x)=\log x+O(\log\log x)$, contradicting Theorem \ref{AsympDiagP}.
\end{proof}

\section*{Acknowledgement}
I would like to express my gratitude to Piotr Miska, who informed me about the problem and helped to improve the quality of the paper. I would also like to thank Carlo Sanna, whose comments allowed me to simplify a part of the paper.

\end{document}